\documentclass[12pt]{article}
\usepackage{tabularx}
\usepackage{graphicx, amssymb, latexsym, amsfonts, amsmath, lscape, amscd, mdframed, ulem,
amsthm, color, epsfig, mathrsfs, tikz, soul, enumerate,verbatim,pst-node,algorithm, algpseudocode}
\usetikzlibrary{decorations.pathmorphing}
\usetikzlibrary{shapes.geometric,calc,shapes,fit,intersections,graphs,graphs.standard,decorations.pathreplacing,decorations.markings,positioning,arrows,shadows,arrows.meta,fadings,decorations.text,3d,snakes,cd}
\usetikzlibrary{arrows,automata}
\usetikzlibrary{graphs}
\usetikzlibrary{graphs.standard}
\tikzset{every state/.style={minimum size=0pt}}

\newtheorem{theorem}{Theorem}[section]

\newtheorem{corollary}[theorem]{Corollary}

\newtheorem{proposition}[theorem]{Proposition}

\theoremstyle{definition}
\newtheorem{definition}[theorem]{Definition}
\newtheorem{construction}[theorem]{Construction}

\theoremstyle{definition}   
\newtheorem{remark}[theorem]{Remark}

\usetikzlibrary{decorations.pathmorphing}
\newcommand\DELETE[1]{}

\usepackage{todonotes}

\begin{document}


\title{{\bf Fundamentals of pushable homomorphisms of oriented graphs}}
\author{
{\sc Tapas Das}$\,^{a}$, {\sc Pavan P D}$\,^{a,c}$, {\sc Sagnik Sen}$\,^{a}$, {\sc S Taruni}$\,^{a,b}$\\
\mbox{}\\
{\small $(a)$ Indian Institute of Technology Dharwad, India.}\\
 {\small $(b)$ Centro de Modelamiento Matemático (CNRS IRL2807),} \\
 {\small Universidad de Chile, Santiago, Chile. }\\
 {\small $(c)$ University of Turku, FI-20014 Turku, Finland,}
}

\date{}

\maketitle

\begin{abstract}
To push a vertex $v$ of an oriented graph $\overrightarrow{G}$ is to reverse the direction of the arcs incident to $v$. 
A homomorphism of an oriented graph $\overrightarrow{G}$ to another oriented graph $\overrightarrow{H}$ is an arc preserving vertex mapping. 
Suppose that it is possible to obtain a $\overrightarrow{G}'$ (called push equivalent to $\overrightarrow{G}$) by pushing some vertices of 
$\overrightarrow{G}$ such that 
$\overrightarrow{G}'$ admits a homomorphism to $\overrightarrow{H}$. 
Then we say that $\overrightarrow{G}$ admits a pushable homomorphism to $\overrightarrow{H}$.  

In this article, we characterize push equivalent orientations of graphs using cycles through presenting the oriented analogue of the well-known Zaslavsky's Lemma from the theory of signed graphs. We then apply our result to show that the problem of determining whether two given orientations of a graph are push equivalent or not can be solved in polynomial time, to count the number of non-push equivalent orientations of a graph, and to provide a canonical definition of pushable homomorphism.  
    
We also provide an one-to-one correspondence between oriented and signed bipartite graphs which commutes with pushable (resp.,  switchable) homomorphisms. In terms of category theory, we prove a categorical isomorphism between the categories of bipartite oriented and signed graphs (where morphisms are pushable and switchable homomorphisms, respectively). Moroever, we apply our correspondence to translate a number of important results directly from the theory of signed graphs to oriented graphs. In particular, we show that pushable homomorphisms of bipartite graphs capture the entire theory of graph coloring as a subcase. 
    
Note that, a complete complexity dichotomy characterization of the 
pushable $\overrightarrow{H}$-coloring problem, that is the problem of determining whether an input oriented graph $\overrightarrow{G}$ admits a pushable homomorphism to $\overrightarrow{H}$, is known. We partially refine the results by showing  that the pushable $\overrightarrow{H}$-coloring problem, where $\overrightarrow{H}$ is a directed odd cycle, is NP-complete even for sparse graphs. On a related note, we also show how the pushable $\overrightarrow{H}$-coloring problem completely encodes graph coloring, where $\overrightarrow{H}$ varies over all directed odd cycles.
\end{abstract}

\noindent \textbf{Keywords:} homomorphisms, oriented colorings,  push operation, oriented graphs, signed graphs.

\section{Introduction and preliminaries}
In 1994, the study of oriented coloring was initiated by Courcelle~\cite{courcelle-monadic} in one of the papers from his series of works which established the illustrious Courcelle's theorem~\cite{courcelle1990monadic}. Following that, Raspaud and Sopena~\cite{planar80} influenced a large body of work (see~\cite{orientedchi, sopena2016homomorphisms} for detailed surveys)   
dedicated to the study and extensions of the theory of oriented coloring. In recent times, oriented coloring is 
primarily studied under the umbrella of homomorphisms of oriented graphs (apart from the surveys~\cite{orientedchi, sopena2016homomorphisms}, one can consult some notable papers:~\cite{borodin2001deeply, borodin2004homomorphisms, marshall17, nandy2016outerplanar, ochem2017homomorphisms}).

The study of homomorphisms of oriented graphs seems to have an uncanny similarity to that of homomorphisms of $2$-edge colored graphs, also known as signed graphs~\cite{naserasr2015homomorphisms}. However, for signed graphs, two particular types of homomorphisms are studied, namely, sign-preserving homomorphisms and switchable homomorphisms\footnote{This particular type of homomorphism has been studied under various different names in literature including ``homomorphism of signed graphs'', ``switching homomorphism", ``switchable homomorphism'', etc.}. While the homomorphisms of oriented graphs is 
analogous to the sign-preserving homomorphism of signed graphs, a recent variant of the former, namely, the pushable homomorphism of oriented graphs is found to be analogous to the switchable homomorphism of signed graphs.
The study of switchable homomorphisms of signed graphs has gained a lot of popularity due to its connections to the theories of graph coloring, circular coloring, graph minors, network flows, and to Ramsey theory~\cite{naserasr2013homomorphisms,naserasr2015homomorphisms,NWZ21,raspaud2011circular}.

Despite having similarities with the well studied theory of switchable homomorphisms~\cite{bensmail2021pushable}, the notion of pushable homomorphisms for oriented graphs are not studied extensively. To this end, we obtain some fundamental results and connections to ordinary graph coloring and the theory of signed graphs. 


\medskip

\noindent \textbf{Oriented graphs and homomorphisms:} 
An \textit{oriented graph} $\overrightarrow{G}$ is a directed graph without any loops or parallel arcs in opposite directions. We denote an oriented graph with an upper case Latin alphabet having an arrow over its head, say $\overrightarrow{G}$, while the underlying undirected graph is denoted by $G$.
We denote the set of vertices and arcs of 
$\overrightarrow{G}$ by $V(\overrightarrow{G})$ and $A(\overrightarrow{G})$, respectively. 
Furthermore, given an arc $uv$, the vertex $u$ is an 
\textit{in-neighbor} of $v$, 
and the vertex $v$ is an 
\textit{out-neighbor} of $u$.
The set of all in-neighbors (resp., out-neighbors) of $u$ is denoted by $N^{-}(u)$ (resp., $N^{+}(u)$).

 A \textit{homomorphism} of an oriented graph $\overrightarrow{G}$  to another oriented graph 
    $\overrightarrow{H}$ is a vertex mapping 
    $f: V(\overrightarrow{G}) \to V(\overrightarrow{H})$ such that 
    for any arc $uv$ of $\overrightarrow{G}$, its image $f(u)f(v)$ is also an arc of $\overrightarrow{H}$. 
If $\overrightarrow{G}$ admits a homomorphism to $\overrightarrow{H}$, then we say that $\overrightarrow{G}$ is \textit{$\overrightarrow{H}$-colorable}, and denote it by 
$\overrightarrow{G} \rightarrow \overrightarrow{H}$. 
A bijective homomorphism such that the images of non-adjacent vertices are also non-adjacent is an \textit{isomorphism}. 
The \textit{oriented chromatic number} of an oriented graph $\overrightarrow{G}$, denoted by $\chi_o(\overrightarrow{G})$, is the minimum $|V(\overrightarrow{H})|$ such that $\overrightarrow{G}$ is 
$\overrightarrow{H}$-colorable.

\medskip

\noindent \textbf{Pushable homomorphisms:} To \textit{push} a vertex $v$ of $\overrightarrow{G}$ is to reverse the direction of all arcs of $\overrightarrow{G}$ that are incident to $v$. We denote the so-obtained oriented graph (after pushing $v$ of $\overrightarrow{G}$) by 
$\overrightarrow{G}^{v}$. Observe that, if we push multiple vertices of $\overrightarrow{G}$, then the order in which we push the vertices does not affect the final resultant graph. Moreover, pushing one vertex two times is the same as not pushing it at all. Therefore, it makes sense to define the push operation for a set of vertices as well. 
Given a set $S \subseteq V(\overrightarrow{G})$, to push $S$ 
is to push each vertex of $S$ exactly once. The so-obtained 
graph is denoted by $\overrightarrow{G}^S$
and we say that $\overrightarrow{G}$ 
and $\overrightarrow{G}^S$ are \textit{push equivalent}, 
denoted by $\overrightarrow{G}\equiv_{p}\overrightarrow{G}^S$. 
Moreover, $[\overrightarrow{G}]$ denotes the push equivalent class of oriented graphs, where $\overrightarrow{G}$ is any representative of it.   That means, given a simple graph $G$, the set of all its orientations is partitioned into push equivalent classes. 

A \textit{pushable homomorphism} of an oriented graph $\overrightarrow{G}$  to another oriented graph $\overrightarrow{H}$ is a vertex mapping $f: V(\overrightarrow{G}) \to V(\overrightarrow{H})$ such that there exists  $\overrightarrow{G}'\equiv_{p}\overrightarrow{G}$ so that $f$ is a homomorphism of $\overrightarrow{G}'$ to $\overrightarrow{H}$. If $\overrightarrow{G}$ admits a pushable homomorphism to $\overrightarrow{H}$, then we say that $\overrightarrow{G}$ is \textit{ pushable $\overrightarrow{H}$-colorable}, and we denote it by $\overrightarrow{G} \xrightarrow{p} \overrightarrow{H}$. Moreover, the \textit{pushable chromatic number} of an oriented graph $\overrightarrow{G}$, denoted by $\chi_p(\overrightarrow{G})$ is the smallest $|V(\overrightarrow{H})|$ such that $\overrightarrow{G}$ is pushably $\overrightarrow{H}$-colorable.

\medskip

\noindent \textbf{A historical note:}
The concept of pushable homomorphisms was first introduced and studied by Klostermeyer and MacGillivray~\cite{push} in 2004. Along with their work, Sen~\cite{sen-push} further explored the relation between the pushable chromatic number and the oriented chromatic number. The complexity dichotomy problems regarding finding the pushable chromatic number were studied in~\cite{ochem-push,push}, and the analogue of clique in this set up was studied in~\cite{nandi-bensmail-push1}. Furthermore, the pushable chromatic number of different graph families, such as, graphs with bounded maximum degree~\cite{bensmail2021pushable}, 
graphs with bounded acyclic chromatic number~\cite{sen-push}, subcubic graphs~\cite{bensmail2021pushable}, planar graphs and planar graphs having girth 
restrictions~\cite{borodin1998universal, das2023pushable, push, sen-push}, sparse graphs defined by maximum average degrees~\cite{bensmail2021pushable,borodin1998universal,das2023pushable}, outerplanar graphs and outerplanar graphs having girth restrictions~\cite{push,sen-push}, and grids~\cite{bensmail2023pushable} have been explored across several papers and by different researchers.

Apart from the above-mentioned works on pushable homomorphisms of oriented graphs, the push operation on digraphs has been studied in different contexts. A few notable works include
the works of  Babai and Cameron~\cite{babai2000automorphisms} on automorphism groups of the pushable (switching) classes of tournaments, of Huang, MacGillivray, and Wood~\cite{huang2001pushing} on characterizing multipartite tournaments which can be made acyclic through pushing a set of vertices, and of Klostermeyer and S\v{o}lt\'es~\cite{klostermeyer1998hamiltonicity} who characterized the multipartite tournaments that can be made Hamiltonian through pushing a set of vertices. 

\begin{remark}
    The problem of determining whether an oriented graph $\overrightarrow{G}$ is pushably $\overrightarrow{H}$-colorable or not can be viewed as a graph modification problem~\cite{chung1994chordal,fomin2013subexponential,villanger2009interval} where the actual problem considered here is 
    determining whether $\overrightarrow{G}$ is $\overrightarrow{H}$-colorable, and the modification is  pushing a vertex subset of $\overrightarrow{G}$. 
\end{remark}

\medskip

\noindent \textbf{Motivation, our contributions, and organization:} In the following we list our 
section-wise contributions, and provide some specific motivation for the same. 
\begin{itemize}
    \item In Section \ref{sec2}, we characterize 
    push equivalent orientations of graphs using cycles. This result is an analogue of the 
    well-known Zaslavsky's Lemma~\cite{zaslavsky1982signed} from the theory of signed graphs. 
    Zaslavsky's Lemma is one of the most important 
    structural results in the rich theory of signed graphs and contributes greatly in building the foundation of this theory. 
    Using our main result, we show that the problem of determining whether two given orientations of a graph are push equivalent or not can be solved in polynomial time, to count the number of non-equivalent orientations of a graph, and to provide a canonical definition of pushable homomorphism.

    \item In Section \ref{sec3}, we provide an one-to-one correspondence between oriented and signed graphs which respect the pushable and the switchable homomorphism orders. In terms of category theory, we prove a categorical isomorphism between the categories of bipartite oriented and signed graphs (where morphisms are pushable and switchable homomorphisms, respectively). Moroever, we apply our main theorem 
    to translate a number of important results directly from the theory of signed graphs to oriented graphs. In particular, we show that pushable homomorphisms of bipartite graphs capture the entire theory of graph coloring as a subcase.

    \item In Section \ref{sec4}, we prove a connection between the graph coloring problem and pushable homomorphism to directed odd cycles. Note that, a complete complexity dichotomy characterization of 
    the decision problem  of determining whether an input oriented graphs 
    $\overrightarrow{G}$ admits a pushable $\overrightarrow{H}$-coloring or not is known due to Klostermeyer and MacGillivray~\cite{push}. We apply our main result to show that the pushable $\overrightarrow{H}$-coloring problem, where $\overrightarrow{H}$ is a directed odd cycle, is NP-complete even for sparse graphs. Also we show how the pushable $\overrightarrow{H}$-coloring problem completely encodes graph coloring.

    \item In Section \ref{sec6}, we share our concluding remarks, open problems, and future research directions. 
\end{itemize}

\medskip

\noindent \textit{Note:} In this article, we follow the book ``Introduction to graph theory'' by West~\cite{D.B.West} for the standard graph theoretic notation and terminology.

\section{Characterization of push equivalent orientations}\label{sec2}
Given a graph $G$, an \textit{ordered closed walk} $C = v_1v_2 \cdots v_kv_1$ is a closed walk together with the prescribed traversal rule which mandates us to traverse the vertex $v_{i+1}$ after $v_{i}$, where $i \in \{1, 2, \cdots, k\}$ and the $+$ operation in the subscript of vertex names is taken modulo $k$.  
Thus, it will make sense to speak about 
\textit{forward arcs} and \textit{backward arcs} of $C$ in this context. 
Notice that, given a closed walk $C = v_1v_2 \cdots v_kv_1$, there can be two distinct ordered closed walks on $C$, that is, (i) the ordered closed walk $v_1v_2 \cdots v_kv_1$  and (ii) the ordered closed walk
$v_1v_kv_{k-1} \cdots v_2v_1$.   The two ordered closed walks are called each other's \textit{conjugates}. 
In particular, a cycle is also a closed walk, thus the above definitions hold for cycles as well. 

An oriented ordered  cycle $\overrightarrow{C}$ 
is \textit{directable} if it is possible to push some vertices of $\overrightarrow{C}$ and make it a directed cycle. 
In particular, if it is possible to push some  vertices of $\overrightarrow{C}$ and make it a directed cycle with all forward (resp., backward) arcs, then $\overrightarrow{C}$ is 
\textit{forward (resp., backward) directable}. Finally, if $\overrightarrow{C}$ is not directable, then it is \textit{non-directable}. 
Based on these definitions, oriented ordered cycles, and by an extension, oriented ordered closed walks can be classified into four types. 
\begin{enumerate}[(i)]
    \item \textit{Odd forward directable closed walk:} An oriented ordered closed walk of odd length that has an odd number of forward arcs, or equivalently, an even number of backward arcs.
    In particular, if the oriented ordered closed walk is an oriented ordered cycle, then the cycle is forward directable.

    \item \textit{Odd backward directable closed walk:} An oriented ordered closed walk of odd length that has an even number of forward arcs, or equivalently, an odd number of backward arcs. In particular, if the oriented ordered closed walk is an oriented ordered cycle, then the cycle is backward directable.

    \item \textit{Even directable closed walk:} An oriented ordered closed walk of even length that has an even number of forward arcs, or equivalently, an even number of backward arcs.
    In particular, if the oriented ordered closed walk is an oriented ordered cycle, then the cycle is directable (both forward and backward).

    \item \textit{Even non-directable closed walk:} An oriented ordered closed walk of even length that has an odd number of forward arcs, or equivalently, an odd number of backward arcs.
    In particular, if the oriented ordered closed walk is an oriented ordered cycle, then the cycle is non-directable. 
\end{enumerate}
If two oriented ordered closed walks are of the same type, then their \textit{directability} is the same while it is different otherwise.

\begin{theorem}\label{th_zas}
    Let $\overrightarrow{G}^{{\tiny 1}}$ and 
    $\overrightarrow{G}^{{\tiny 2}}$ be two orientations of the graph $G$. The two orientations of $G$ are push equivalent if and only if every ordered cycle of $G$ has the same directability in both $\overrightarrow{G}^{{\tiny 1}}$
    and $\overrightarrow{G}^{{\tiny 2}}$. 
\end{theorem}

\begin{proof}
It is enough to prove the statement when $G$ is a connected graph. 

Note that, pushing the vertices of an oriented ordered cycle does not change the parity of the number of its forward (resp., backward) arcs. Thus, if $\overrightarrow{G}^{{\tiny 1}}$ and $\overrightarrow{G}^{{\tiny 2}}$ 
are push equivalent, any ordered cycle of $G$ has the same directability in $\overrightarrow{G}^{{\tiny 1}}$ and $\overrightarrow{G}^{{\tiny 2}}$. This proves the ``only if'' part of the statement.

To prove the ``if'' part, let us assume that any ordered cycle of $G$ has the same directability in 
$\overrightarrow{G}^{{\tiny 1}}$ and $\overrightarrow{G}^{{\tiny 2}}$. 
Take a spanning tree $T$ of $G$. 
Let $\overrightarrow{T}^{{\tiny 1}}$ and $\overrightarrow{T}^{{\tiny 2}}$ 
be the induced orientations of $T$ under $\overrightarrow{G}^{{\tiny 1}}$ and 
$\overrightarrow{G}^{{\tiny 2}}$, respectively. Suppose that $uv$ is an arc of 
 $\overrightarrow{T}^{{\tiny 1}}$ while $vu$ is an arc of 
 $\overrightarrow{T}^{{\tiny 2}}$. Notice that, $uv$ is a cut edge of $T$ (since $T$ is a tree). Let $A$ and $B$ be the sets of vertices of the two different connected components of $T - uv$, respectively. Observe that, if we push all vertices of $A$ in 
 $\overrightarrow{T}^{{\tiny 2}}$, then the arc $vu$ changes its direction while all the other arcs of 
 $\overrightarrow{T}^{{\tiny 2}}$ retain their direction. Thus, we can repeat this process for all edges of $T$ where the direction of their respective arcs in $\overrightarrow{T}^{{\tiny 1}}$ and $\overrightarrow{T}^{{\tiny 2}}$ are different. 
 As a result, we will obtain a push equivalent orientation 
 $\overrightarrow{G}^{{\tiny 3}}$ of $\overrightarrow{G}^{{\tiny 2}}$ in which the tree $T$ has exactly the orientation  
 $\overrightarrow{T}^{{\tiny 1}}$. As the directability of an oriented ordered cycle is invariant under pushing, the directability of all cycles of $\overrightarrow{G}^{{\tiny 1}}$
and 
$\overrightarrow{G}^{{\tiny 3}}$ must also be the same. As any edge $u'v'$ from $E(G) \setminus E(T)$ is part of a cycle in $G$ (since $T$ is a spanning tree), the arcs corresponding to the edge $u'v'$ in 
$\overrightarrow{G}^{{\tiny 1}}$
and 
$\overrightarrow{G}^{{\tiny 3}}$ must have the same direction. Therefore, $\overrightarrow{G}^{{\tiny 1}}$
and 
$\overrightarrow{G}^{{\tiny 3}}$ must be the same orientation of $G$ which implies that $\overrightarrow{G}^{{\tiny 1}}$
is push equivalent to 
$\overrightarrow{G}^{{\tiny 2}}$. 
\end{proof}

\begin{remark}
   In the classification of ordered closed walks presented before the statement of Theorem~\ref{th_zas}, in each item, the last line is a remark that maybe considered as an observation. These observations also follow directly from Theorem~\ref{th_zas}. 
\end{remark}

 The general notion of homomorphism (not necessarily of graphs) usually preserves some important property of the object for which it is defined. In the case of pushable homomorphisms, it preserves the directability of ordered cycles according to the above theorem. However, from the  given definition, it is not so easy to realise which properties get preserved.  
 Hence, the following alternative definition can be regarded 
 as the canonical definition of pushable homomorphism.   

\begin{definition}[The canonical definition]
    A vertex mapping $f: V(\overrightarrow{G}) \to V(\overrightarrow{H})$ is a pushable homomorphism of $\overrightarrow{G}$ to $\overrightarrow{H}$ if
    $f$ preserves the directability of each ordered closed walk of $\overrightarrow{G}$. 
\end{definition}

Apart from motivating and justifying the canonical definition of a pushable homomorphism, the main result of this section, that is, Theorem~\ref{th_zas}, 
has a number of other interesting consequences. 
An immediate corollary is the following.

\begin{corollary}
\label{cor:forest}
Any two orientations of a forest $F$ are push equivalent.  
\end{corollary}

\begin{proof}
As a forest $F$ does not have any cycle, 
according to Theorem~\ref{th_zas}, 
any two orientations of $F$ must be push equivalent. 
\end{proof}

Moreover, Theorem~\ref{th_zas} implies a polynomial algorithm 
for deciding push equivalence between two orientations of a graph.

\begin{proposition}
    Let $\overrightarrow{G}^{{\tiny 1}}$ and 
    $\overrightarrow{G}^{{\tiny 2}}$ be two orientations of the graph $G$.
    There is a polynomial time algorithm to determine whether 
    $\overrightarrow{G}^{{\tiny 1}}$ is push equivalent to 
    $\overrightarrow{G}^{{\tiny 2}}$. 
\end{proposition}

\begin{proof}
    Note that there are well-known~\cite{D.B.West} linear time algorithms (for example, Kruskal's algorithm)  for finding a spanning tree 
    of a graph. 
    Therefore,     we can apply the procedure used in the proof of Theorem~\ref{th_zas}
    to devise a polynomial algorithm. 
\end{proof}

Furthermore, we can count the number of distinct orientations up to push equivalence of a graph. 

\begin{proposition}
    Let $G$ be a graph with $n$ vertices, $m$ edges, and $c$ components. Then the number of 
    orientations of $G$, that are not push equivalent to each other, is $2^{m-n+c}$. 
\end{proposition}

\begin{proof}
    Let $F$ be a spanning forest of $G$. As any two orientations of $F$ are push equivalent due to Corollary~\ref{cor:forest}, the distinct orientations of $G$ will be determined by the orientations of the edges that do not belong to $F$. We know that a forest on $n$ vertices having $c$ components have exactly $(n-c)$ edges. Thus, there will be $2^{m-n+c}$ distinct choices for orienting the edges of $G$ that do not belong to $F$. 
\end{proof}

Notice that the above count is for a labelled graph $G$ and does not take isomorphism into account. 
In contrast, next, we are going to list the oriented graphs that remain invariant (up to isomorphism) under the push operation. 

\begin{theorem}
    Let $\overrightarrow{G}$ be a connected oriented graph that is isomorphic to each of its push equivalent oriented graphs. 
    Then $\overrightarrow{G}$ is either
    the one vertex (oriented) graph $\overrightarrow{K}_1$, the orientation $\overrightarrow{K}_2$ 
    of a complete graph on two vertices, or 
    the non-directable $4$-cycle.  
    \end{theorem}

\begin{proof}
Let $\overrightarrow{G}$ be an oriented graph such that $|V(\overrightarrow{G})| \leq 3$. It is trivial to see that $\overrightarrow{K}_1$ and $\overrightarrow{K}_2$ are isomorphic to their push equivalent oriented graphs. Furthermore, there are no such oriented graphs on three vertices since the transitive triangle and directed triangle are push equivalent but not isomorphic. 

Observe that the non-directable $4$-cycle is isomorphic to all its push equivalent oriented graphs. 
Let $\overrightarrow{G}$ be a connected oriented graph with 
$|V(\overrightarrow{G})| \geq 4$
which is not equal to the non-directable $4$-cycle. 
It is enough to find two non-isomorphic push equivalent orientations of $\overrightarrow{G}$.

If $G$ has diameter $1$, that is, if $G$ is a complete graph, then push all in-neighbors of a particular vertex $v$ of $\overrightarrow{G}$ to obtain its push equivalent oriented graph $\overrightarrow{G}^{{\tiny 1}}$. Notice that,  $\overrightarrow{G}^{{\tiny 1}}$ has a source, and no sink. Next, push all out-neighbors of  $v$ of $\overrightarrow{G}$ to obtain its push equivalent oriented graph $\overrightarrow{G}^{{\tiny 2}}$. Note that, $\overrightarrow{G}^{{\tiny 2}}$ has a sink, and no source. Hence, $\overrightarrow{G}^{{\tiny 1}}$ and $\overrightarrow{G}^{{\tiny 2}}$ are two  
non-isomorphic push equivalent orientations of $\overrightarrow{G}$.

If $G$ has diameter at least $3$, then we can take a spanning tree $T$ of $G$ and push some vertices of $\overrightarrow{G}$ so that the orientation of $T$ has a single source $v$ (say). This is possible due to Corollary~\ref{cor:forest}. Let us denote this push equivalent orientation of $\overrightarrow{G}$ by $\overrightarrow{G}^{{\tiny 1}}$.
Observe that, $\overrightarrow{G}^{{\tiny 1}}$ has at most one source in it.
On the other hand, choose two vertices $v_1$ and $v_2$ of $G$ that are at a distance equal to $3$. Then push the in-neighbors of $v_1$ and $v_2$ to obtain the 
push equivalent oriented graph $\overrightarrow{G}^{{\tiny 2}}$ of $\overrightarrow{G}$.
Notice that, both $v_1$ and $v_2$ are sources in 
$\overrightarrow{G}^{{\tiny 2}}$. In particular, that means $\overrightarrow{G}^{{\tiny 2}}$ has at least two sources, and thus cannot be isomorphic to 
$\overrightarrow{G}^{{\tiny 1}}$.

If $G$ has diameter $2$, then it can have at most one cut vertex.
If $G$ has a cut vertex $v$, then we can push all in-neighbors of $v$ in $\overrightarrow{G}$ to obtain its push equivalent oriented graph $\overrightarrow{G}^{{\tiny 1}}$. On the other hand, we can push 
all out-neighbors of $v$ in $\overrightarrow{G}$ to obtain its push equivalent oriented graph $\overrightarrow{G}^{{\tiny 2}}$.
Note that, the cut vertex $v$ is a source in 
$\overrightarrow{G}^{{\tiny 1}}$ and a sink in 
$\overrightarrow{G}^{{\tiny 2}}$. Thus, $\overrightarrow{G}^{{\tiny 1}}$ cannot be isomorphic to 
$\overrightarrow{G}^{{\tiny 2}}$ as $v$ must map to $v$ under any isomorphism.

Hence, suppose that $G$ has diameter $2$
and is
$2$-connected. 
Further, if $G$  has a 
directable cycle $C$, 
then consider the graph 
$G' = G[V(G) \setminus V(C)]$ induced by the vertices which are not part of $C$. Take a spanning forest $F$ of $G'$, and assume that the components of $F$ are $T_1, T_2, \cdots, T_r$. As $G$ is connected and has diameter $2$, there is a vertex $v_i$ in $T_i$ which is adjacent to a vertex of $C$, for each $i \in \{1,2,\cdots, r\}$.  Now push the vertices of $\overrightarrow{G}$ to obtain an orientation $\overrightarrow{G}^{{\tiny 1}}$ where $C$ is oriented as a directed cycle,
$T_i$ is oriented as a 
directed tree with a single source $v_i$ (source inside the tree) for all $i \in \{1, 2, \cdots, r\}$, and $v_i$ has an in-neighbor from $V(C)$.  
Note that, $\overrightarrow{G}^{{\tiny 1}}$ does not have any source.  On the other hand, pick any vertex $v$ of $\overrightarrow{G}$ and push all its in-neighbors to obtain the push equivalent orientation  $\overrightarrow{G}^{{\tiny 2}}$ with $v$ as a source. As $\overrightarrow{G}^{{\tiny 1}}$ does not have any source, and $\overrightarrow{G}^{{\tiny 2}}$ has at least one source $v$, they cannot be isomorphic.

Therefore, as all odd cycles are directable (forward or backward), 
the only case that we are yet to consider is when $G$ is a $2$-connected bipartite graph of diameter $2$ whose cycles are all non-directable. Therefore, $G$ must be a complete bipartite graph $K_{m,n}$, where $n \geq 3$ and $m \geq 2$. 
However, it is not possible to obtain an orientation of 
$K_{2,3}$ so that all its cycles are non-directable. Thus, such a graph $\overrightarrow{G}$ cannot exist. This completes the proof. 
\end{proof}

\section{Connection with signed graphs}\label{sec3}
In the following, we establish a connection between pushable homomorphisms of oriented graphs and switchable homomorphisms of signed graphs. 
The theory of signed graphs has been well studied~\cite{beck2006number, harary, naserasr2022density, naserasr2015homomorphisms, naserasr2021homomorphisms, zaslavsky1982signed} and it is deemed important due to its strong connections with graph coloring, graph minor theory, Ramsey Theory, flows, etc. However, let us recall some basics of signed graphs to keep our paper self contained.

\medskip

\noindent \textbf{Preliminaries of signed graphs:} 
A \textit{signed graph} $(G, \sigma)$ is a graph $G$ along with a \textit{signature function} $\sigma: E(G) \to \{+, -\}$ that assigns a positive or a negative sign to the edges of $G$. 
A \textit{sign-preserving homomorphism} of a 
signed graph $(G, \sigma)$ to  another signed graph 
    $(H, \pi)$  is a vertex mapping 
    $f: V(G) \to V(H)$ such that 
    for any edge $uv$ of $(G, \sigma)$, its image $f(u)f(v)$ is also an edge of $(H, \pi)$ satisfying $\sigma(uv) = \pi(f(u)f(v))$.

    To \textit{switch} a vertex $v$ of $(G, \sigma)$ is to 
    change the sign of all the edges incident to $v$. 
    Given a set $S \subseteq V(G)$, to switch $S$ 
is to push each vertex of $S$ exactly once. 
The so-obtained signed
graph is denoted by $(G, \sigma^S)$
and we say that $(G, \sigma)$ 
and $(G, \sigma^S)$ are \textit{switch equivalent}, 
denoted by $(G, \sigma) \equiv_{s} (G, \sigma^S)$. 
Moreover, $[(G, \sigma)]$ denotes the switch equivalent class of signed graphs, where $(G, \sigma)$ is any representative of it.   That means, given a simple graph $G$, the set of all of its signatures is partitioned into switch equivalent classes.

    A \textit{switchable homomorphism} of a 
    signed graph $(G, \sigma)$  to another signed graph 
    $(H, \pi)$ is a vertex mapping 
    $f: V(G) \to V(H)$ such that 
    there exists a switch equivalent graph $(G, \sigma^S)$ of 
    $(G, \sigma)$ so that $f$ is a sign-preserving homomorphism of 
    $(G, \sigma^S)$ to $(H, \pi)$. 
    If $(G, \sigma)$ admits a 
    switchable homomorphism to $(H, \pi)$, then we 
    denote it by $(G, \sigma) \xrightarrow{s} (H, \pi)$.
    Moreover, the \textit{switchable chromatic number} of a
    signed graph $(G, \sigma)$, denoted by $\chi_s((G, \sigma))$ is the smallest $|V(H)|$ such that $(G, \sigma) \xrightarrow{s} (H, \pi)$.

\medskip

Let $\mathcal{C}_p$ (resp.,  $\mathcal{C}_s$) denote the set of all 
push (resp., switch) equivalence classes of oriented (resp., signed) bipartite graphs. 
We say $f: V(G) \to V(H)$ is a \textit{morphism} of $[\overrightarrow{G}]$ to $[\overrightarrow{H}]$  
(resp., of $[(G, \sigma)]$ to  $[(H, \pi)]$) 
if $f: \overrightarrow{G} \xrightarrow{p} \overrightarrow{H}$
(resp., $f: (G, \sigma) \xrightarrow{s} (H, \pi)$) is a pushable (resp., switchable) homomorphism.

\begin{theorem}\label{Th 1-2-1}
There exists a bijection 
$\Phi: \mathcal{C}_p \to \mathcal{C}_s$ such that $f$ is a morphism of 
 $[\overrightarrow{G}]$ to $[\overrightarrow{H}]$ 
 if and only if 
 $f$ is a morphism of 
 $\Phi([\overrightarrow{G}]) = [(G, \sigma)]$ to  $\Phi([\overrightarrow{H}]) = [(H, \pi)]$. 
\end{theorem}

\begin{proof}
    Firstly, we describe the function $\Phi$ through a construction. Let $\overrightarrow{G}$ be a bipartite oriented graph with 
    partite sets $A$  and $B$ (in this order). Thus, all arcs of $\overrightarrow{G}$ 
    have exactly one endpoint in $A$ and the other in $B$. Now we construct a signed graph 
    $(G, \sigma_{(\overrightarrow{G},A,B)})$ in the following manner: 
    we replace each arc from $A$ to $B$ with a positive edge and each arc from $B$ to $A$ with a negative edge. Now we set 
    $\Phi([\overrightarrow{G}]) = [(G, \sigma_{(\overrightarrow{G},A,B)})]$. 
    Notice that, the choice of $A$ and $B$ is not unique unless the graph $G$ is connected. Also, even if $G$ is connected, swapping the 
    order of the partite sets $A$ and $B$ would produce a different signed graph 
    $(G, \sigma_{(\overrightarrow{G},B,A)})$. Therefore, it is important to prove that $\Phi$ is well-defined.

Let $G$ be a bipartite graph with partite sets $A, B$ (resp., $A',B'$). Note that, for showing that 
$\Phi$ is well-defined it is enough to prove 
$(G, \sigma_{(\overrightarrow{G},A,B)})$
and 
$(G, \sigma_{(\overrightarrow{G}^S,A',B')})$ 
are switch equivalent, where $S$ is a vertex subset of $G$.  To do so, we will find a way to switch
a set of vertices of 
$(G, \sigma_{(\overrightarrow{G},A,B)})$ to obtain 
$(G, \sigma_{(\overrightarrow{G}^S,A',B')})$. 

As a connected bipartite graph has a unique bipartition, for any connected component of $G$,
either $A \cap V(C) = A' \cap V(C)$ and 
$B \cap V(C) = B' \cap V(C)$ or 
$A \cap V(C) = B' \cap V(C)$ and 
$B \cap V(C) = A' \cap V(C)$. 
Let 
$A \cap A' = A_1$, 
$B \cap B' = B_1$,
$A \cap B' = A_2$, 
and 
$B \cap A' = B_2$. 
Thus, 
it is possible to express $G$ as a disjoint union of $G_1$ and $G_2$ such that $A_1, B_1$ are partite sets of $G_1$ and $A_2$, $B_2$ are partite sets of $G_2$.  Therefore, if we switch the vertices of $A_2$ in $(G, \sigma_{(\overrightarrow{G},A,B)})$, we will obtain the graph 
$(G, \sigma_{(\overrightarrow{G},A',B')})$. 
Next observe that, it is possible to obtain the graph 
$(G, \sigma_{(\overrightarrow{G}^S,A',B')})$ from 
$(G, \sigma_{(\overrightarrow{G},A',B')})$ by simply switching the set $S$ of vertices. 
This proves that $\Phi$ is well-defined.

Next we will show that the function $\Phi$ is injective. Suppose that $\Phi([\overrightarrow{G}]) = \Phi([\overrightarrow{G}'])$. To show $\Phi$ is injective, we need to show that $\overrightarrow{G}'$ is push equivalent to $\overrightarrow{G}$. 
First of all, note that $\overrightarrow{G}$ and $\overrightarrow{G}'$ must have the same underlying graph $G$ for $\Phi([\overrightarrow{G}]) = \Phi([\overrightarrow{G}'])$ to hold. Let $A, B$ be partite sets of $G$. 
Thus, in particular we know that 
$(G, \sigma_{(\overrightarrow{G}, A, B)})$  and 
$(G, \sigma_{(\overrightarrow{G}', A, B)})$ are switch equivalent. 
Suppose 
that 
$(G, \sigma_{(\overrightarrow{G}', A, B)})$ can be obtained from 
$(G, \sigma_{(\overrightarrow{G}, A, B)})$ by switching the set $S$ of vertices. Thus, if we switch the vertices of $S$ in $\overrightarrow{G}$, we will obtain the oriented graph 
$\overrightarrow{G}'$, proving 
$\overrightarrow{G} \equiv_p \overrightarrow{G}'$. Hence, $\Phi$ is indeed injective.

To prove that $\Phi$ is a bijection, we are left to show that 
$\Phi$ is a surjection. Let $[(G,\sigma)]$ be a switch equivalence class of signed bipartite graphs. Moreover, let $A, B$ be partite sets of $G$. 
Let us describe a construction of an oriented graph based on $(G, \sigma)$. 
The oriented graph $\overrightarrow{G}$ is obtained by replacing all the positive edges with arcs from $A$ to $B$ and all the negative edges with arcs from $B$ to 
$A$. Notice that, we will have $(G, \sigma_{(\overrightarrow{G}, A, B)}) = (G, \sigma)$. That means, we have found an oriented graph $\overrightarrow{G}$ satisfying $\Phi([\overrightarrow{G}]) = [(G, \sigma)]$. 
Hence, the function $\Phi$ is surjective, and thus, bijective.

Finally, suppose that $f:V(G) \to V(H)$ is a morphism of 
$[\overrightarrow{G}]$ to $[\overrightarrow{H}]$, This essentially means that there exists a homomorphism of an element of $[\overrightarrow{G}]$ to an element of $[\overrightarrow{H}]$. Without loss of generality, we may assume that there exists a homomorphism of 
$\overrightarrow{G}$ to $\overrightarrow{H}$. Next, let us fix a bipartition $(A,B)$ of $G$. Furthermore, we will fix a bipartition $(A', B')$ of $H$ satisfying the properties 
$f(A) \subseteq A'$ and $f(B) \subseteq B'$. It is possible to find such a bipartition $(A', B')$ of $H$ as both $G,H$ are bipartite graphs. Now note that, $f$ is a 
sign-preserving homomorphism of 
$(G, \sigma_{(\overrightarrow{G}, A, B)})$ to $(H, \sigma_{(\overrightarrow{H}, A', B')})$. That implies, $f$ is a morphism of $[\Phi(\overrightarrow{G})]$ to $[\Phi(\overrightarrow{H})]$. This proves the ``only if'' part of the statement. The ``if'' part of the statement can be proved similarly.  
\end{proof}

\begin{remark}
    In terms of category theory, if we consider categories where objects are push equivalence classes of bipartite oriented graphs and switch equivalent classes of  bipartite signed graphs,  while pushable and switchable homomorphisms play the roles of morphisms, respectively, then the above mentioned one-to-one correspondence 
    is an isomorphic covariant functor (and its inverse) which shows that the two categories are isomorphic. This is probably the underlying reason why we find similarities between the theory of pushable homomorphisms of oriented graphs and switchable homomorphisms of signed graphs. 
\end{remark}

\subsection{Applications}
We begin by showing how the entire theory of graph coloring can be viewed as a special case of pushable 
homomorphisms of oriented bipartite graphs.

\medskip  
Given a graph $G$, we construct an oriented bipartite graph $\overrightarrow{S}(G)$ as follows: we replace each edge $uv$ by a $4$-cycle of the type 
 $uw_1vw_2u$ and orient the cycle to form a non-directable $4$-cycle. 
 
 Let $A, B$ be the two partite sets of the complete bipartite graph $K_{n,n}$. Let $\overrightarrow{K}_{n,n}^*$ be the orientation of $K_{n,n}$ in which a perfect matching is oriented from $A$ to $B$ and the rest of the edges are oriented from $B$ to $A$.

 Then, the following result is an analogue of one by Naserasr, Rollova, and Sopena~\cite{naserasr2015homomorphisms}.

\begin{theorem}\label{Th sg}
    Let $G$ be a graph. Then, $\chi(G) \leq k$ if and only if $S(\overrightarrow{G}) \xrightarrow{p} \overrightarrow{K}_{k,k}^*$.
\end{theorem}

\begin{proof}
    Let $\hat{S}(G)$ be the signed graph obtained by replacing each edge $uv$ of $G$ by a (signed) $4$-cycle 
    $uw_1vw_2u$ having three positive edges and one negative edge. 
    Let $(K_{k,k}, \tau)$ be the signed graph whose underlying graph is $K_{k,k}$, and the set of negative edges is a perfect matching. 
    
    Notice that, 
    $\hat{S}(G) \in \Phi([S(\overrightarrow{G})])$
    and 
    $(K_{k,k}, \tau) \in \Phi([\overrightarrow{K}_{n,n}^*])$. 
    We know due to Naserasr, Rollova, and Sopena (see Theorem 6.2 in~\cite{naserasr2015homomorphisms}) that  
    $\chi(G) \leq k$ if and only if 
    $\hat{S}(G) \xrightarrow{s} (K_{k,k}, \tau)$. Thus we are done using Theorem~\ref{Th 1-2-1}. 
\end{proof}

Observe that using Theorem~\ref{Th sg}, the equivalent formulation of the Four-Color Theorem will be: all $S(\overrightarrow{G})$ admits a pushable homomorphism to $\overrightarrow{K}_{4,4}^*$, where $G$ is a planar graph. Note that, the oriented graph $S(\overrightarrow{G})$ is thus an oriented planar bipartite graph, and furthermore, not all oriented planar bipartite graphs can be expressed as 
$S(\overrightarrow{G})$ for some planar $G$. 
Therefore, the next result  is stronger than the Four-Color Theorem. We will prove it using 
Theorem~9.5 from~\cite{naserasr2015homomorphisms}, whose proof uses the Four-Color Theorem. 

\begin{theorem}
Every planar oriented bipartite graph admits a pushable homomorphism to $\overrightarrow{K}_{4,4}^*.$
\end{theorem}

\begin{proof}
    The signed graph $(K_{4,4}, \tau)$ is obtained by assigning negative signs to the edges of a perfect matching in $K_{4,4}$ and by assigning positive signs to the rest of the edges. We know that 
    planar signed bipartite graphs admit a homomorphism to $(K_{4,4}, \tau)$ due to Naserasr, Rollov{\'a} and Sopena~\cite{naserasr2015homomorphisms}. The proof is completed by 
    using Theorem~\ref{Th sg} as 
    $\overrightarrow{K}^*_{4,4} \in \Phi^{-1}([(K_{4,4}, \tau)])$.
   \end{proof}

The Folding Lemma by
Klostermeyer and Zhang \cite{klostermeyer20002+} is very useful to study the homomorphism properties of a planar graph. This lemma claims that for each planar graph $G$ of shortest odd cycle
length $2r + 1$, and for each $k \leq r$, there is a planar homomorphic image $H$ of $G$ where every face of
$H$ is of length $2k + 1$ and the shortest odd-length cycle of $H$ is also of length $2k + 1$. By considering
negative cycles instead of odd-length cycles Naserasr, Rollova and Sopena~\cite{naserasr2013homomorphisms}, proved 
a similar result for the class of planar signed
bipartite graphs. The analogous result for oriented graphs uses a new notion of ``balance of an even cycle'' for oriented graphs. 

An oriented even ordered cycle $\overrightarrow{C}_{2n}$
on $2n$ vertices 
is \textit{balanced} if there exists $S \subseteq V(\overrightarrow{C}_{2n})$ such that 
$\overrightarrow{C}_{2n}^S$ has exactly $n$ forward and $n$ backward arcs. Notice that in the definition of balanced, whether
$\overrightarrow{C}_{2n}$ is ordered or not is redundant.  Furthermore, an oriented even cycle is \textit{unbalanced} if it is not balanced. By Theorem~\ref{th_zas}, one can observe that $\overrightarrow{C}_{2n}$
is balanced if and only if either $n$ is even and $\overrightarrow{C}_{2n}$ is directable or $n$ is odd and $\overrightarrow{C}_{2n}$ is 
non-directable. 
The length of a shortest balanced (resp., unbalanced) cycle of an oriented bipartite graph is its \textit{balanced} 
(resp., \textit{unbalanced}) girth.

\begin{theorem}
Let $\overrightarrow{G}$ be a planar oriented bipartite graph with unbalanced girth $g$. If
$C = v_0 \cdots v_{r-1}v_0$ is a balanced facial cycle of $\overrightarrow{G}$, or an unbalanced facial cycle of $\overrightarrow{G}$ with
$r>g$, then there is an integer 
$i \in \{0,\cdots,r-1\}$ such that the oriented graph $\overrightarrow{G}'$ obtained from $\overrightarrow{G}$ by identifying $v_{i-1}$ and $v_{i+1}$ (subscripts are taken modulo r) is a pushable homomorphic image of $\overrightarrow{G}$
of unbalanced girth $g$.
\end{theorem}

\begin{proof}
A signed cycle is positive (resp., negative) if it has even (resp., odd) number of negative edges. 
Observe that, if $\overrightarrow{C}$ is an even balanced (resp., unbalanced) oriented cycle, then any $(C, \sigma) \in \Phi([\overrightarrow{C}])$ is a positive (resp., negative) signed cycle. 
It is known (see  Lemma $3.1$ in~\cite{naserasr2013homomorphisms})  
that the exact statement is true if we replace oriented by signed, balanced by positive, and unbalanced by negative. Thus, the proof follows directly applying  
Theorem~\ref{Th 1-2-1}.     
\end{proof}

The repeated application of this result, gives us the following.

\begin{theorem}
Given a planar oriented bipartite graph $\overrightarrow{G}$ of unbalanced girth $g$, there is a pushable homomorphic image $\overrightarrow{G}'$ of $\overrightarrow{G}$ such that:
\begin{enumerate}[(i)]
    \item $\overrightarrow{G}'$ is planar,
    \item $\overrightarrow{G}'$ is an oriented bipartite graph,
    \item $\overrightarrow{G}'$ is of unbalanced girth $g$,
    \item every face of $\overrightarrow{G}'$ is an unbalanced cycle of length $g$.
\end{enumerate}
\end{theorem}

\begin{proof}
  In~\cite{naserasr2013homomorphisms}, the equivalent result for signed graphs is proved (see Corollary $3.2$ in~\cite{naserasr2013homomorphisms}). 
  Therefore, the proof follows applying Theorem~\ref{Th 1-2-1}. 
\end{proof}

The next result is the oriented graph analogue of one due to Naserasr, Pham and Wang~\cite{naserasr2022density}, which was proved using the 
potential method~\cite{kostochka2014ore}.
An oriented graph $\overrightarrow{G}$ is  \textit{pushably $\overrightarrow{H}$-critical} if there does not exist a pushable homomorphism of $\overrightarrow{G}$ to
$\overrightarrow{H}$ but every proper subgraph of $\overrightarrow{G}$ admits a pushable homomorphism to
$\overrightarrow{H}$.

\begin{figure}[h]

    \begin{tabularx}{\textwidth}{ X X}
				\centering
    \scalebox{1}{
    \begin{tikzpicture}[scale=0.9,auto=center]
    
  \node[fill=white,circle,draw] (a1) at (0,0) {\scriptsize$v_1$};  
  \node[fill=white,circle,draw] (a2) at (2,0)  {\scriptsize$v_2$};   
  \node[fill=white,circle,draw] (a3) at (4,0)  {\scriptsize$v_3$};  
  \node[fill=white,circle,draw] (a4) at (6,0) {\scriptsize$v_4$};

  \node[fill=white,circle,draw] (b1) at (0,-2.5) {\scriptsize$u_1$};  
  \node[fill=white,circle,draw] (b2) at (2,-2.5)  {\scriptsize$u_2$};   
  \node[fill=white,circle,draw] (b3) at (4,-2.5)  {\scriptsize$u_3$};

  \node at (2,-3.5) {$(i)$};

  \draw[-,thick, blue] (a1) -- (b1);
  \draw[-,thick, blue] (a1) -- (b2);
  \draw[-,thick, blue] (a1) -- (b3);
  \draw[dashed ,thick, red] (a2) -- (b1);
  \draw[-,thick, blue] (a2) -- (b2);
  \draw[dashed,thick, red] (a3) -- (b2);
  \draw[-,thick, blue] (a3) -- (b3);
  \draw[-,thick, blue] (a4) -- (b1);
  \draw[dashed,thick, red] (a4) -- (b3);
    
\end{tikzpicture}}

&

    \centering
    \scalebox{1}{
    \begin{tikzpicture}[scale=0.9,auto=center]
    
  \node[fill=white,circle,draw] (a1) at (0,0) {\scriptsize$v_1$};  
  \node[fill=white,circle,draw] (a2) at (2,0)  {\scriptsize$v_2$};   
  \node[fill=white,circle,draw] (a3) at (4,0)  {\scriptsize$v_3$};  
  \node[fill=white,circle,draw] (a4) at (6,0) {\scriptsize$v_4$};

  \node[fill=white,circle,draw] (b1) at (0,-2.5) {\scriptsize$u_1$};  
  \node[fill=white,circle,draw] (b2) at (2,-2.5)  {\scriptsize$u_2$};   
  \node[fill=white,circle,draw] (b3) at (4,-2.5)  {\scriptsize$u_3$};

  \node at (2,-3.5) {$(ii)$};

  \draw[->,thick] (a1) -- (b1);
  \draw[->,thick] (a1) -- (b2);
  \draw[->,thick] (a1) -- (b3);
  \draw[<-,thick] (a2) -- (b1);
  \draw[->,thick] (a2) -- (b2);
  \draw[<-,thick] (a3) -- (b2);
  \draw[->,thick] (a3) -- (b3);
  \draw[->,thick] (a4) -- (b1);
  \draw[<-,thick] (a4) -- (b3);
    
\end{tikzpicture}}
    \label{fig:enter-label}

    \end{tabularx}

     \caption{$(i)$ The signed graph $(W, \tau)$. $(ii)$ The oriented graph $\overrightarrow{W}$.}
     \label{fig:enter-label1111}
\end{figure}

\begin{theorem}
    Let $\overrightarrow{C}_4$ be the unbalanced directed $4$-cycle.  
    If $\overrightarrow{G}$ is a pushably $\overrightarrow{C_4}$-critical oriented graph which is not isomorphic to $\overrightarrow{W}$ depicted in Fig.~\ref{fig:enter-label1111}$(ii)$, then
    $$|A(\overrightarrow{G})| \geq \frac{4|V(\overrightarrow{G})|}{3}.$$
\end{theorem}

\begin{proof}
    Let $(C_4, \epsilon)$ denote the signed graph whose underlying graph is a 
    $4$-cycle, and it has exactly one negative edge. Notice that $\Phi([\overrightarrow{C}_4]) = [(C_4, \epsilon)]$ and $\Phi([\overrightarrow{W}]) = [(W, \tau)]$ (see Fig.~\ref{fig:enter-label1111}.

    Let $(G, \sigma)$ be a signed graph that 
    does not admit a switchable homomorphism to $(C_4, \epsilon)$, but any of its proper subgraph does. It is known (see Theorem~$3.8$ in~\cite{naserasr2022density}) that
    $|E(G)| \geq \frac{4|V(G)|}{3}$ unless 
    $(G, \sigma)$ is switchably isomorphic to $(W, \tau)$. Notice that, the only graphs that can admit a homomorphism to $(C_4, \epsilon)$ (resp., $\overrightarrow{C}_4$) 
    are bipartite graphs. Therefore, the proof follows using Theorem~\ref{Th 1-2-1}. 
\end{proof}

The equivalent statement of the above theorem for signed graphs is important for multiple reasons~\cite{naserasr2022density}. Interested readers are suggested to consult~\cite{naserasr2022density} for further details. 

\hspace{2cm}

\begin{remark}
    The pushable chromatic number of oriented hexagonal grids and square grids are studied (see Theorems 3.1, 3.4, 4.1, 4.2 in~\cite{bensmail2023pushable}). On the other hand, 
    the chromatic number signed hexagonal grids and square grids are studied (see Theorem~7 in~\cite{jacques2021chromatic} and Theorems 10, 12 in~\cite{dybizbanski2020signed}). One can note, both hexagonal and square grids are bipartite graphs. Thus, due to Theorem~\ref{Th 1-2-1} one can conclude that  Theorems~3.1, 3.4 in~\cite{bensmail2023pushable} (resp., Theorems 4.1, 4.2 in~\cite{bensmail2023pushable}) are equivalent to 
    Theorems 7 in~\cite{jacques2021chromatic} (resp., Theorems 10, 12 in~\cite{dybizbanski2020signed}). 
\end{remark}

\section{Pushable homomorphisms to directed odd cycles}\label{sec4}
In this section, we analyse the complexity of the problem of determining whether an input oriented graph 
$\overrightarrow{G}$ admits a pushable homomorphism to an oriented odd cycle. Notice that, the complete complexity dichotomy of this problem is known due to Klostermeyer and MacGillivray~\cite{push} for general (unrestricted families) oriented graphs. However, in our proofs, the input graph is from restricted (sparse) graph families, and hence in the cases of pushable homomorphisms to oriented odd cycles, our result is a (partial) refinement to the findings of Klostermeyer and MacGillivray~\cite{push}.  We state the decision problem as follows. 

\begin{mdframed}
\noindent
\textsc{Pushable $\overrightarrow{H}$-Coloring}

\noindent
\textbf{Instance:} A graph $\overrightarrow{G}$.

\noindent
\textbf{Question:} Does there exist a pushable homomorphism of $\overrightarrow{G}$ to $\overrightarrow{H}$?
\end{mdframed}

Before going to our complexity results, let us first prove a general correspondence between the pushable homomorphism to directed odd cycles and graph coloring. 
A simple graph $G$ is \textit{$k$-colorable} if it is possible to label the vertices of $G$ using a set of $k$ colors in such a way that adjacent vertices of $G$ receive distinct labels. The decision version of the problem is as follows. 

\begin{mdframed}
\noindent
\textsc{$k$-Coloring}

\noindent
\textbf{Instance:} A graph $G$.

\noindent
\textbf{Question:} Is $G$ $k$-colorable? 
\end{mdframed}

It is well-known  that 
the \textsc{$k$-Coloring} problem is 
NP-complete for 
$k \geq 3$~\cite{hell}.

\begin{construction}\label{cons push-color}
Given a simple graph $G$, and any integer $k \geq 1$, we construct an oriented graph as follows. 
Each edge $uv$ in $G$ is replaced with two internally disjoint oriented paths $\overrightarrow{P}$  and $\overrightarrow{P'}$ of length $4k$ that connect $u$ and $v$. The parities of forward arcs (with respect to traversal from $u$ to $v$) in $\overrightarrow{P}$ and 
$\overrightarrow{P'}$ are even and odd, respectively. Let the so-obtained oriented graph be $\overrightarrow{G}^{(k)}$. Notice that there are several ways to obtain $\overrightarrow{G}^{(k)}$ from $G$. 
Irrespective of that, we get the same $\overrightarrow{G}^{(k)}$ up to push equivalence.  Thus, in the context of pushable homomorphism, the construction of $\overrightarrow{G}^{(k)}$ from $G$ is 
well-defined.
\end{construction}

\medskip

\begin{theorem}\label{thm:2k+1col}
    For any $k \geq 1$, a simple graph $G$ is $(2k+1)$-colorable if and only if $\overrightarrow{G}^{(k)}$ admits a pushable homomorphism to the directed cycle $\overrightarrow{C}_{2k+1}$. 
\end{theorem}

\begin{proof} 
Notice that the portion corresponding to an edge $uv$ of $G$ in $\overrightarrow{G}^{(k)}$ is two internally disjoint oriented paths $\overrightarrow{P}$ and $\overrightarrow{P}'$ connecting the vertices $u$ and $v$. 
Moreover, $\overrightarrow{P}$ (resp., $\overrightarrow{P}'$) has even (resp., odd) number of forward and backward arcs, with respect to traversal from $u$ to $v$. For convenience, let us denote this particular induced subgraph of $\overrightarrow{G}^{(k)}$ by $\overrightarrow{X}$. 

Notice that, if we push some of the internal vertices of $\overrightarrow{P}$ 
(resp., $\overrightarrow{P}'$), 
the parity of the number of forward arcs remain unchanged. Since there are exactly $4k-1$ internal vertices in $\overrightarrow{P}$ (resp., $\overrightarrow{P}'$), it is possible to obtain 
$2^{4k-1}$ different orientations of $P$ 
having even (resp., odd) number of forward and backward arcs. 
On the other hand, as there are 
exactly $2^{4k} = 2 \cdot 2^{4k-1}$ different orientations of $P$ (since it has $4k$ edges),
we can say that it is possible to obtain any orientation of $P$ having even (resp., odd) number of forward and backward arcs by pushing only the internal vertices of $\overrightarrow{P}$ 
(resp., $\overrightarrow{P}'$).

If we push one of the vertices among $u$ and $v$ in $\overrightarrow{X}$, the oriented path $\overrightarrow{P}$ 
(resp., $\overrightarrow{P}'$) 
gets modified to an oriented path having odd (resp., even) number of forward and backward arcs. That means, pushing exactly one of $u$ and $v$ 
exchanges the role of $\overrightarrow{P}$ and $\overrightarrow{P}'$ in $\overrightarrow{X}$.

Let $\overrightarrow{C}_{2k+1} = w_0w_1w_2\cdots w_{2k}w_0$ be the directed cycle on $2k+1$ vertices. 
Next we want to show that  it is possible to find a pushable homomorphism 
$f: \overrightarrow{X} \xrightarrow{p} \overrightarrow{C}_{2k+1}$  with $f(u) = w_i$ and $f(v) = w_j$ if and only if $i \neq j$, where  $i,j \in \{0, 1, 2, \cdots, 2k\}$. 
Notice that, this local information actually can be implemented globally. That is, if the above is correct, then we can say that $\overrightarrow{G}^{(k)}$ admits a pushable homomorphism to $\overrightarrow{C}_{2k+1}$
if and only if it is possible to label the vertices of $G$ using the vertices of 
$\overrightarrow{C}_{2k+1}$ in such a way that adjacent vertices obtain different labels. This is 
essentially equivalent to the statement of the theorem we are trying to prove. Therefore, it is enough to show that 
given a partial function 
$f: V(\overrightarrow{X}) \to V(\overrightarrow{C}_{2k+1})$ defined by $f(u) = w_i$ and $f(v) = w_j$, it is possible to extend  $f$ to a pushable homomorphism of $\overrightarrow{X}$
to $\overrightarrow{C}_{2k+1}$ if and only if $i \neq j$. 
Notice that,  it is the same as obtaining a push equivalent 
orientation $\overrightarrow{X}'$ that admits a homomorphism to $\overrightarrow{C}_{2k+1}$ which is an extension of $f$ if and only if $i \neq j$. Moreover, since pushing the vertex $u$ or $v$ 
only toggles the roles of $\overrightarrow{P}$
and $\overrightarrow{P}'$, if such a $\overrightarrow{X}'$ exist, without loss of generality we may assume that it is possible to obtain $\overrightarrow{X}'$ from $\overrightarrow{X}$ without pushing the vertices $u$ and $v$.
We will now formulate this problem algebraically.

Suppose that the oriented paths $\overrightarrow{P}$ and $\overrightarrow{P}'$ got modified to $\overrightarrow{P}_1$ and $\overrightarrow{P}'_1$, respectively. Moreover suppose that there are exactly $b$ (a positive even number less than or equal to $4k$) backward arcs in $\overrightarrow{P}_1$ and $b'$ (a positive odd number less than $4k$) backward arcs in $\overrightarrow{P}'_1$. Therefore, it is possible to extend $f$ to a homomorphism of $\overrightarrow{X}'$ to $\overrightarrow{C}_{2k+1}$ if and only if the following equations are satisfied:
\begin{equation}
    i + (4k-b) - b \equiv j~(\bmod~2k+1),
\end{equation}\label{equation even}
\begin{equation}
    i + (4k-b') - b' \equiv j~(\bmod~2k+1).
\end{equation}\label{equation odd}
Since $b$ (resp., $b'$) is an even (resp., odd) number less than or equal to $4k$, we may assume that $b = 2a$ and $b' = 2a'+1$ where 
$a \in \{0, 1, \cdots, 2k\}$ and $a' \in \{0, 1, \cdots, 2k-1\}$. Moreover, we may also assume that 
$(j-i) = \ell$ is some integer between $0$ to $2k$.  
Thus, the modified Equations~(\ref{equation even}) and~(\ref{equation odd}) will be as follows. 
\begin{equation}
4k-4a = 4(k-a) \equiv \ell~(\bmod~2k+1),
\end{equation}\label{equation even modified}
\begin{equation}
4k-4a'-2 = 4(k-a')-2 \equiv \ell~(\bmod~2k+1).
\end{equation}\label{equation odd modified}
As $4$ and $2k+1$ are co-primes and $a$ varies among all integers from $0$ to $2k$, the numbers of the form $4(k-a)~(\bmod~2k+1)$ gives us $2k+1$ distinct numbers modulo $2k+1$.  
To be precise,   Equation~(\ref{equation even modified}) has a solution for all $\ell \in \{0, 1, \cdots, 2k\}$. 

Similarly, as $2$ and
$2k+1$ are co-primes and $a$ varies among all integers from $1$ to $2k-1$, the numbers of the form $4(k-a')-2~(\bmod~2k+1)$ gives us $2k$ distinct numbers modulo $2k+1$.   
To be precise,   Equation~(\ref{equation odd modified}) has a solution for all $\ell \in \{0, 1, \cdots, 2k\}$ except one value.  
Notice that,
for $\ell =0$, Equation~(\ref{equation odd modified}) 
will have any solution if and only if there is a solution for one of the following equations:
\begin{equation}\label{equation l=0}
    4k=4a'+2. 
\end{equation}

\begin{equation}\label{equation l=2k+1}
    4k-4a'-2 = \pm (2k+1).
\end{equation}
However, Equation~(\ref{equation l=0}) doesn't have a solution as the left hand side is not divisible by $4$, whereas the right hand side of it is. Furthermore, Equation~(\ref{equation l=2k+1}) does not have a solution as the left hand side is even while the right hand side is odd. Thus, the proof is completed. 
\end{proof}

\begin{corollary}
    Let $G^*$ be a graph obtained by adding a new vertex $z$ to $G$ which is adjacent to all vertices of $G$ and let  $\overrightarrow{G}^{(k)}$ be the oriented graph due to Construction~\ref{cons push-color}. Then the following are equivalent
    for all $k \geq 3$:
    \begin{enumerate}[(i)]
        \item The graph $G$ is $(2k+1)$-colorable. 

        \item The graph $G^*$ is $(2k+2)$-colorable. 

        \item The oriented graph $\overrightarrow{G}^{(k)}$ admits a pushable homomorphism to the directed cycle $\overrightarrow{C}_{2k+1}$ of order $2k+1$. 
    \end{enumerate}
\end{corollary}

\begin{proof}
    The equivalence of (i) and (ii) is trivial while the equivalence of (i) and (iii) follows directly from Theorem~\ref{thm:2k+1col}. 
\end{proof}

\begin{remark}
    In the theory of homomorphisms of graphs, homomorphism to an odd cycle having $2k+1$ vertices corresponds to circular 
    $\left(2+\frac{1}{k}\right)$-coloring~\cite{zhu2001circular}, 
    which is integral in the study of Jagear’s
Conjecture (disproved) and it's restriction for planar graphs (open)~\cite{jaeger1984circular}. Coincidentally, pushable homomorphisms of oriented graphs to oriented odd cycles 
    turn out to be of special interest since it captures the notion of $k$-coloring for all 
    $k \geq 3$. 
\end{remark}

Next we are going to show that the 
\textsc{Pushable $\overrightarrow{C}_{2k+1}$-Coloring} 
is an NP-complete problem even for sparse graphs.

\begin{theorem}
    The \textsc{Pushable $\overrightarrow{C}_{2k+1}$-Coloring} problem is NP-complete for the family of bipartite graphs having girth at least $8k$, where $\overrightarrow{C}_{2k}$ is any oriented cycle on $2k$ vertices.  
\end{theorem}

\begin{proof}
  Let $G$ be a simple graph with $n$ vertices and $m$ edges. Then notice that $\overrightarrow{G}^{(k)}$ has $n+2m(4k-1)=n+8mk-2m$ vertices and $8km$ arcs. That means, the number of vertices and arcs of 
  $\overrightarrow{G}^{(k)}$ is $O(n+m)$. 
  Moreover, observe that the underlying graph of $\overrightarrow{G}^{(k)}$ is a bipartite graph with girth at least $8k$. 
  It is well-known that the \textsc{$k$-Coloring} problem is NP-complete for all $k \geq 3$. Thus, using the correspondence established in Theorem~\ref{thm:2k+1col}, we can conclude this proof. 
\end{proof}

\section{Conclusions}\label{sec6}
\noindent (1) While Zaslavsky's Lemma acts as the
backbone for the theory of signed graphs, one of the 
major reasons why signed graphs have a deep connection 
to graph minor theory is due to its definition of minor~\cite{naserasr2015homomorphisms} 
that preserves the signs of the cycles. 
Since we already have an oriented analogue of the Zaslavsky's Lemma (Theorem~\ref{th_zas}), is it possible 
to have an analogue of ``signed minor'' in oriented 
graphs that interacts ``nicely'' with the push 
operation? The term ``nicely'' will vaguely refer to preserving directability of ordered cycles. 

\medskip

\noindent (2) There are two ``consistent'' types of signed graphs that play important roles in the theory of homomorphisms of signed graphs~\cite{naserasr2015homomorphisms}: 
odd signed graphs (possible to switch some vertices and make all edges negative) and 
bipartite signed graphs (underlying graphs are bipartite)~\cite{naserasr2015homomorphisms}. Due to Theorem~\ref{Th 1-2-1} proved in this article, we have seen that the exact oriented analogues of the latter are bipartite oriented graphs. However, what 
a natural oriented analogue of the former could be, remains an open question.

\medskip

\noindent (3) There are several notions that are 
(or could be) defined for oriented graphs, 
but not for signed graphs such as 
oriented arc coloring and oriented chromatic index, 
oriented total coloring and oriented total chromatic number, 
deeply critical oriented graphs~\cite{ochem2008oriented, bensmail2024oriented, borodin2001deeply}, etc. 
On the other hand, many signed graph notions 
like $0$-free-coloring, 
circular chromatic number,
flows~\cite{NWZ21, raspaud2011circular}, etc. 
are yet to be translated in the theory of oriented graphs. It seems to be a challenging task to, firstly,  find analogues of these notions, and secondly, justify why the notion is a natural analogue. However, Theorem~\ref{Th 1-2-1} addresses the second issue partially. We know that any analogous notions in the theory of oriented graphs and signed graphs must behave the exact same way when restricted to the class of bipartite graphs. Therefore, it will be an interesting task to look for suitable analogues of notions to enrich both the theories. 

\medskip

\noindent (4) It will be interesting to study oriented graphs with respect to pushable homomorphisms as a graph category. 
To be precise, let $\mathcal{C}^*_p$ be the category in which oriented graphs play the role of objects and pushable homomorphisms play the role of morphisms. As a special case of the works done in~\cite{sen2022homomorphisms}, we know that categorical product and co-product exists for $\mathcal{C}_p$. Restricting two general open questions asked in~\cite{sen2022homomorphisms}, we would like to enquire that: (i) does there exist an analog of the notion of exponential graphs~\cite{hell2004graphs} in $\mathcal{C}_p$, (ii) does there exist an analog of the density theorem~\cite{hell2004graphs} in $\mathcal{C}_p$. Both these questions are also open in the context of signed graphs. 

\medskip

\noindent (5) Since the complete dichotomy characterization for the 
\textsc{pushable $\overrightarrow{H}$-Coloring}
problem is known due to Klostermeyer and MacGillivray~\cite{push}, it is interesting to study the complexity of the problem for restricted classes of graph. We did this for sparse graphs where our $\overrightarrow{H}$ was directed odd cycles. However, one may choose a different (family) $\overrightarrow{H}$ and study the complexity of the problem for various graph classes. In particular, it will be interesting to have polynomial time solution for the problem for important graph (sub)classes such as graphs having low maximum degree, maximum average degree, edge density, treewidth, etc. 

\medskip

\noindent \textbf{Acknowledgements:}  This work is partially supported by SERB-MATRICS ``Oriented chromatic and clique number of planar graphs'' (MTR/2021/000858), Academy of Finland grant number 338797 and Centro de Modelamiento Matemático (CMM) BASAL fund FB210005 for center of excellence from ANID-Chile.

\bibliographystyle{abbrv}
\bibliography{reference.bib}

\end{document}